\documentclass[10pt]{amsart}

\usepackage{times,amsfonts,amsmath,amstext,amsbsy,amssymb,
  amsopn,amsthm,upref,eucal, amscd}
\usepackage[T1]{fontenc}
\usepackage{color}

\usepackage[pdftex]{graphicx}

\newtheorem{theorem}{Theorem}[section]
\newtheorem{lemma}[theorem]{Lemma}

\newtheorem{proposition}[theorem]{Proposition}

\numberwithin{equation}{section}

\theoremstyle{definition}

\newtheorem{remark}[theorem]{Remark}

\def\leq{\leqslant }
\def\geq{\geqslant}

\begin{document}

\title[On Freiman's isolated points in $M\setminus L$]{Markov spectrum near Freiman's isolated points in $M\setminus L$}

\author[C. Matheus and C. G. Moreira]{Carlos Matheus and Carlos Gustavo Moreira}

\address{Carlos Matheus:
Universit\'e Paris 13, Sorbonne Paris Cit\'e, CNRS (UMR 7539),
F-93430, Villetaneuse, France.
}

\email{matheus.cmss@gmail.com}

\address{Carlos Gustavo Moreira:
IMPA, Estrada Dona Castorina 110, 22460-320, Rio de Janeiro, Brazil
}

\email{gugu@impa.br}

\date{\today}

\begin{abstract}
Freiman proved in 1968 that the Lagrange and Markov spectra do not coincide by exhibiting a countable infinite collection $\mathcal{F}$ of isolated points of the Markov spectrum which do not belong the Lagrange spectrum. 

In this paper, we describe the structure of the elements of the Markov spectrum in the largest interval $(c_{\infty}, C_{\infty})$ containing $\mathcal{F}$ and avoiding the Lagrange spectrum. In particular, we compute the smallest known element $f$ of $M\setminus L$, and we show that the Hausdorff dimension of the portion of the Markov spectrum between $c_{\infty}$ and $C_{\infty}$ is $> 0.2628$. 
\end{abstract}
\maketitle

\tableofcontents

\section{Introduction}

\subsection{Statement of the main results} The study of Diophantine approximation problems naturally led Markov to investigate (in 1880) two closed subsets $L\subset M$ of the real line called the Lagrange and Markov spectra. After that, a vast literature dedicated to these spectra was developed, and the reader is encouraged to consult the book \cite{CF} of Cusick--Flahive for an excellent introduction to this fascinating topic. 

Freiman \cite{Fr68} showed in 1968 that $M\setminus L\neq\emptyset$ by exhibiting a number $\sigma\simeq 3.1181\dots\in M\setminus L$. In the same article, Freiman also explained how to modify the construction of $\sigma$ in order to obtain a countable infinite collection $\mathcal{F}$ of isolated points of $M$ which are not in $L$. 

In this paper, we exploit the techniques in our previous article \cite{MaMo} to describe the structure of the intersection of $M\setminus L$ with the largest interval $(c_{\infty}, C_{\infty})$ containing $\mathcal{F}\ni \sigma$ which is disjoint from $L$. As a consequence, we show that:

\begin{theorem}\label{t.A} The Hausdorff dimension of $(M\setminus L)\cap (c_{\infty}, C_{\infty})$  satisfies:
$$0.2628 < HD((M\setminus L)\cap (c_{\infty}, C_{\infty}))$$
\end{theorem}

Another consequence of our arguments is the construction of the smallest \emph{known} number in  $M\setminus L$:

\begin{proposition}\label{p.new-number} The smallest element of $(M\setminus L)\cap (c_{\infty}, C_{\infty})$ is
$$f = \frac{71788723850 + 2 \sqrt{210}}{101867079581}+ \frac{217 + \sqrt{156817}}{254} = 3.11812017815984\dots$$
\end{proposition}

\subsection{Organization of the article} After some preliminary discussions in Section \ref{s.preliminaries}, we state in Section \ref{s.HD(M-L)>0} a refinement of Theorem \ref{t.A} saying that $HD((M\setminus L)\cap (c_{\infty}, C_{\infty})) = HD(Y)$, where $Y$ is a Cantor set of real numbers in $[0,1]$ whose continued fraction expansions correspond to the elements of $\{1,2\}^{\mathbb{N}}$ not containing twenty seven explicit finite words. In particular, this reduces the proof of Theorem \ref{t.A} to the computation of lower bounds on $HD(Y)$. Next, we devote Sections \ref{s.forbidden}, \ref{s.allowed}, \ref{s.characterization} and \ref{s.proof} to the derivation of several ingredients needed for the proof of the equality $HD((M\setminus L)\cap (c_{\infty}, C_{\infty})) = HD(Y)$. After that, we explain in Section \ref{a.PT} how to establish a lower bound on $HD(Y)$ ultimately leading to the proof of Theorem \ref{t.A}. In Section \ref{a.new-numbers}, we pursue the arguments in Section \ref{s.characterization} in order to establish Proposition \ref{p.new-number}. Finally, we show in Appendix \ref{a.Berstein} that $(c_{\infty}, C_{\infty})$ is the largest interval disjoint from $L$ containing $\sigma$: in particular, we correct some claims made by Berstein in Theorem 1 at page 47 of \cite{Be73} concerning $(c_{\infty}, C_{\infty})$.  

\section{Some preliminaries}\label{s.preliminaries}

\subsection{Continued fractions} For $\alpha\in\mathbb{R}\setminus\mathbb{Q}$, let 
$$\alpha=[a_0; a_1, a_2,\dots] = a_0+\frac{1}{a_1+\frac{1}{a_2+\frac{1}{\ddots}}}$$
be its continued fraction expansion.

Recall that given $\alpha=[a_0; a_1,\dots, a_n, a_{n+1},\dots]$ and $\beta=[a_0; a_1,\dots, a_n, b_{n+1},\dots]$ with $a_{n+1}\neq b_{n+1}$, one has $\alpha>\beta$ if and only if $(-1)^{n+1}(a_{n+1}-b_{n+1})>0$. 

\subsection{Lagrange and Markov spectra (after Perron)} Given  $A=(a_n)_{n\in\mathbb{Z}}\in(\mathbb{N}^*)^{\mathbb{Z}}$ and $i\in\mathbb{Z}$, let 
$$\lambda_i(A) := [a_i; a_{i+1}, a_{i+2}, \dots] + [0; a_{i-1}, a_{i-2}, \dots]$$

In 1921, Perron proved that
$$L=\{\ell(A)<\infty: A\in(\mathbb{N}^*)^{\mathbb{Z}}\} \quad \textrm{and} \quad M=\{m(A)<\infty: A\in(\mathbb{N}^*)^{\mathbb{Z}} \}$$
where 
$$\ell(A):=\limsup\limits_{i\to\infty}\lambda_i(A) \quad \textrm{and} \quad m(A) := \sup\limits_{i\in\mathbb{Z}} \lambda_i(A)$$ 
are the Lagrange and Markov values of $A$. 

We will deal exclusively with these characterizations of $L$ and $M$ in this paper. 

\subsection{Gauss-Cantor sets} The \emph{Gauss-Cantor set} associated to a finite alphabet $B=\{\beta_1,\dots,\beta_m\}$, $m\geq 2$, consisting of finite words $\beta_j\in(\mathbb{N}^*)^{r_j}$, $1\leq j\leq m$, such that $\beta_i$ does not begin by $\beta_j$ for all $i\neq j$, is 
$$K(B):=\{[0;\gamma_1, \gamma_2,\dots]: \gamma_i\in B \,\,\,\, \forall \, i\geq 1\}\subset [0,1]$$

\subsection{Some notations} Given a finite word $(b_1,\dots, b_r)\in(\mathbb{N}^*)^r$, we denote its \emph{transpose} by $(b_1,\dots, b_r)^T:=(b_r,\dots,b_1)$. 

We shall indicate periodic continued fractions and bi-infinite sequences which are periodic in one or both sides by putting a bar over the period: for instance, $[\overline{2,1,1,2}] = [2;1,1,2,2,1,1,2,\dots]$ and $\overline{1}, 2, 1, 2, \overline{1, 1, 2} = \dots, 1, 1, 1, 2, 1, 2, 1, 1, 2, 1, 1, 2, \dots$

Moreover, we shall use subscripts to indicate the multiplicity of a digit in a sequence: for example, $[2; 1_3, 2_2, 1, 2, \dots] = [2; 1, 1, 1, 2, 2, 1, 2, \dots]$.

\section{Computation of $HD((M\setminus L)\cap(c_{\infty}, C_{\infty}))$}\label{s.HD(M-L)>0}

In 1968, Freiman \cite{Fr68} showed that
$$\sigma:=\lambda_0(S):=[\overline{2_4, 1_2, 2_2, 1}] + [0; 1, 2_2, 1_2, 2_4, 1, 2_2, 1_2, 2_2, 1_2, \overline{2_2, 1_2, 2_2, 1, 2_2}]\in M \setminus L$$

In the sequel, we shall revisit Freiman's arguments in order to prove the following result. Let $Y$ be the Cantor set
\begin{equation}\label{e.Cantor-Y}
Y:=\{[0;\gamma]:\gamma\in\{1,2\}^{\mathbb{N}} \textrm{ not containing the subwords in } P\}
\end{equation}
where $P$ is the finite set of 27 words consisting of the words (1) to (13) in Lemma \ref{l.F1} below and their transposes, and the words $2 1_2 2_4 1 2_2 1_2 2_3$ and its transpose. 

Also, let
$$c_{\infty} := [\overline{2_4,1_2,2_2,1}]+[0;\overline{1,2_2,1_2,2_4}] = 3.11812017814369\dots$$
and
\begin{eqnarray*}
C_{\infty} &:=& [2; 1, 2_2, 1_2, 2_4, 1, 2_2, 1_2, 2_4, 1, 2_2, 1_2, 2_2, 1, 2_4, 1_2, \overline{2_3, 1_3}]
+ [0; 2_3, 1_2, 2_2, 1, 2_4, 1,\overline{1_3, 2_3}] \\
&=& 3.118120178328746016\dots
\end{eqnarray*}

\begin{theorem}\label{t.M-L-piece-HD} $HD((M\setminus L)\cap (c_{\infty}, C_{\infty})) = HD(Y)$ (where $Y$ is the Cantor set in \eqref{e.Cantor-Y}).
\end{theorem}

The next four sections are devoted to the proof of this result.


\section{Forbidden strings}\label{s.forbidden}

\begin{lemma}\label{l.F1} If $B\in\{1, 2\}^{\mathbb{Z}}$ contains any of the strings
\begin{itemize}
\item[(1)] $12^*1$
\item[(2)] $2_2 1 2^* 2_2 1$
\item[(3)] $2_3 1 2^* 2_2$
\item[(4)] $2_4 2^* 1 2_2 1_2$
\item[(5)] $2_3 1_2 2_3 2^* 1 2_2 1_2 2_3$
\item[(6)] $2_2 1_2 2_3 2^* 1 2_2 1_2 2_3 1$
\item[(7)] $1_2 2_2 1_2 2_3 2^* 1 2_2 1_2 2_3$
\item[(8)] $2_2 1 2_2 1_2 2_3 2^* 1 2_2 1_2 2_5$
\item[(9)] $1 2_2 1 2_2 1_2 2_3 2^* 1 2_2 1_2 2_4 1$
\item[(10)] $2_4 1 2_2 1_2 2_3 2^* 1 2_2 1_2 2_4 1_2$
\item[(11)] $2_4 1 2_2 1_2 2_3 2^* 1 2_2 1_2 2_4 1 2_2 1 2$
\item[(12)] $1 2_4 1 2_2 1_2 2_3 2^* 1 2_2 1_2 2_4 1 2_2 1_3$
\item[(13)] $2 1 2_4 1 2_2 1_2 2_3 2^* 1 2_2 1_2 2_4 1 2_2 1_2 2_2$
\end{itemize}
then $\lambda_j(B)>3.1181201786$ or $\lambda_{j+11}(B)>3.15$ where $j$ is the position in asterisk. 
\end{lemma}

\begin{proof} If $B$ contains (1), then $\lambda_j(B) \geq [2; 1,\overline{1,2}] + [0; 1, \overline{1,2}] > 3.15$. If $B$ contains (2), then $\lambda_j(B) \geq [2; 2_2, 1, \overline{1,2}] + [0; 1, 2_2, \overline{2,1}] > 3.12$. 

\medskip 

If $B$ contains (3), then $\lambda_j(B) \geq [2; 2_2,\overline{2,1}] + [0; 1, 2_3, \overline{2,1}] > 3.119$. If $B$ contains (4), then $\lambda_j(B) \geq [2; 1, 2_2, 1_2, \overline{1,2}] + [0; 2_4, \overline{2,1}] > 3.1182$. 

\medskip 

If $B$ contains (5), then $\lambda_j(B) \geq [2; 1, 2_2, 1_2, 2_3, \overline{2,1}] + [0; 2_3, 1_2, 2_3,  \overline{2,1}] > 3.118125$. If $B$ contains (6), then $\lambda_j(B) \geq [2; 1, 2_2, 1_2, 2_3, 1, \overline{1,2}] + [0; 2_3, 1_2, 2_2, \overline{1,2}] > 3.118121$. 

\medskip 

If $B$ contains (7), then $\lambda_j(B) \geq [2; 1, 2_2, 1_2, 2_3, \overline{2,1}] + [0; 2_3, 1_2, 2_2, 1_2, \overline{1,2}] > 3.118121$. If $B$ contains (8), then $\lambda_j(B) \geq [2; 1, 2_2, 1_2, 2_5,  \overline{2,1}] + [0; 2_3, 1_2, 2_2, 1, 2_2, \overline{2,1}] > 3.1181206$. 

\medskip 

If $B$ contains (9) and $\lambda_{j+11}(B)\leq 3.15$, then the discussion of (1) above implies that $\lambda_j(B) \geq [2; 1, 2_2, 1_2, 2_4, 1, 2, 2, \overline{2,1}] + [0; 2_3, 1_2, 2_2, 1, 2_2, 1, \overline{1,2}] > 3.1181202$.

\medskip 

If $B$ contains (10), then $\lambda_j(B) \geq [2; 1, 2_2, 1_2, 2_4, 1_2, \overline{1,2}] + [0; 2_3, 1_2, 2_2, 1, 2_4, \overline{2,1}] > 3.1181202$. If $B$ contains (11), then $\lambda_j(B) \geq [2; 1, 2_2, 1_2, 2_4, 1, 2_2, 1, 2,  \overline{2,1}] + [0; 2_3, 1_2, 2_2, 1, 2_4, \overline{2,1}] > 3.1181201787$.

\medskip 

If $B$ contains (12), then $\lambda_j(B) \geq [2; 1, 2_2, 1_2, 2_4, 1, 2_2, 1_3, \overline{1,2}] + [0; 2_3, 1_2, 2_2, 1, 2_4, 1, \overline{1,2}] > 3.1181201786$. If $B$ contains (13), then $\lambda_j(B) \geq [2; 1, 2_2, 1_2, 2_4, 1, 2_2, 1_2, 2_2,  \overline{2,1}] + [0; 2_3, 1_2, 2_2, 1, 2_4, 1, 2, \overline{2,1}] > 3.1181201789$.
\end{proof}

\section{Allowed strings}\label{s.allowed}

\begin{lemma}\label{l.F2} If $B\in\{1, 2\}^{\mathbb{Z}}$ contains any of the strings
\begin{itemize}
\item[(15)] $1_2 2^* 2$
\item[(16)] $2_2 1 2^* 2 1$
\item[(17)] $2_3 2^* 1 2_2 1 2$
\item[(18)] $1 2_3 2^* 1 2_2 1_3$
\item[(19)] $1 2_3 2^* 1 2_2 1_2 2_2 1$
\item[(20)] $2 1 2_3 2^* 1 2_2 1_2 2_3$
\item[(21)] $1_3 2_3 2^* 1 2_2 1_2 2_3$
\end{itemize}
then $\lambda_j(B)<3.118117$ or $\lambda_{j-6}(B)>3.15$ where $j$ is the position in asterisk. 
\end{lemma}

\begin{proof} If $B$ contains (15), $\lambda_j(B)\leq [2; 2, \overline{2,1}] + [0; 1_2, \overline{1,2}] < 3.05$. If $B$ contains (16), $\lambda_j(B)\leq [2; 2, 1, \overline{1,2}] + [0; 1, 2_2, \overline{2,1}] < 3.09$. If $B$ contains (17), $\lambda_j(B)\leq [2; 1, 2_2, 1, 2, \overline{2,1}] + [0; 2_3, \overline{2,1}] < 3.118$. If $B$ contains (18), $\lambda_j(B)\leq [2; 1, 2_2, 1_3, \overline{1,2}] + [0; 2_3, 1, \overline{1,2}] < 3.118$. 

\medskip 

If $B$ contains (19) and $\lambda_{j-6}(B)\leq 3.15$, then Lemma \ref{l.F1} (1) implies that $\lambda_j(B)\leq [2;1,2_2,1_2,2_2,1,\overline{1,2}] + [0;2_3, 1,1,2,2,\overline{2,1}] < 3.118117$. 

\medskip

If $B$ contains (20), $\lambda_j(B)\leq [2;1,2_2,1_2,2_3,\overline{1,2}] + [0;2_3, 1, 2, \overline{2,1}] < 3.118$. If $B$ contains (21), $\lambda_j(B)\leq [2;1,2_2,1_2,2_3,\overline{1,2}] + [0;2_3, 1_3,\overline{1,2}] < 3.11801$. 
\end{proof}

\section{Sequences with Markov values in $(c_\infty, C_{\infty})$}\label{s.characterization}

\begin{lemma}\label{l.F4} Let $B\in\{1,2\}^{\mathbb{Z}}$ such that $3.118117<\lambda_0(B)$ and $\lambda_n(B)<3.1181201786$ for $n\in\{0,\pm2, \pm6,\pm9,\pm11,\pm15\}$. Then, $B_{-14}\dots B_{16}$ or $(B_{-16}\dots B_{14})^T$ equals to 
$$1_2 2_4 1 2_2 1_2 2_4 1 2_2 1_2 2_4 1 2_2 1_2 2_2$$
\end{lemma}

\begin{proof} After performing a transposition if necessary, we see that Lemma \ref{l.F1} (1), Lemma \ref{l.F2} (15) and our assumption on $\lambda_0(B)$ imply  
$$B_1B_0B_1B_2 = 2 2^* 12$$ 

By Lemma \ref{l.F1} (1) and our assumption on $\lambda_{\pm2}(B)$, we get  
$$B_1B_0B_1B_2B_3 = 2 2^* 122$$

In view of our assumption on $\lambda_0(B)$, by successively applying Lemma \ref{l.F2} (16), Lemma \ref{l.F1} (2), Lemma \ref{l.F1} (3), Lemma \ref{l.F2} (17), Lemma \ref{l.F1} (4), Lemma \ref{l.F2} (18), we deduce that 
$$B_{-4}\dots B_6 = 1 2_3 2^* 1 2_2 1_2 2$$

By Lemma \ref{l.F1} (1) and our assumption on $\lambda_{\pm6}(B)$, we get  
$$B_{-4}\dots B_7 = 1 2_3 2^* 1 2_2 1_2 2_2$$

In view of our assumption on $\lambda_0(B)$, by successively applying Lemma \ref{l.F2} (19), (20), (21), we obtain that 
$$B_{-6}\dots B_8 = 2 1_2 2_3 2^* 1 2_2 1_2 2_3$$

By Lemma \ref{l.F1} (1) and our assumption on $\lambda_{\pm6}(B)$, we get  
$$B_{-7}\dots B_8 = 2_2 1_2 2_3 2^* 1 2_2 1_2 2_3$$

In view of our assumption on $\lambda_0(B)$, by successively applying Lemma \ref{l.F1} (5), (6), (7), we derive that 
$$B_{-9}\dots B_9 = 2 1 2_2 1_2 2_3 2^* 1 2_2 1_2 2_4$$

By Lemma \ref{l.F1} (1) and our assumption on $\lambda_{\pm9}(B)$, we obtain   
$$B_{-10}\dots B_9 = 2_2 1 2_2 1_2 2_3 2^* 1 2_2 1_2 2_4$$ 

In view of our assumption on $\lambda_0(B)$, by successively applying Lemma \ref{l.F1} (8), (9), we get  
$$B_{-11}\dots B_{10} = 2_3 1 2_2 1_2 2_3 2^* 1 2_2 1_2 2_4 1$$

In view of our assumption on $\lambda_{\pm9}(B)$, by successively applying Lemma \ref{l.F1} (2), (4), we get  
$$B_{-13}\dots B_{10} = 1 2_4 1 2_2 1_2 2_3 2^* 1 2_2 1_2 2_4 1$$

By Lemma \ref{l.F1} (10) and our assumption on $\lambda_{0}(B)$, we obtain   
$$B_{-13}\dots B_{11} = 1 2_4 1 2_2 1_2 2_3 2^* 1 2_2 1_2 2_4 1 2$$ 

In view of our assumption on $\lambda_{\pm11}(B)$, by successively applying Lemma \ref{l.F1} (1), (3), we get  
$$B_{-13}\dots B_{13} = 1 2_4 1 2_2 1_2 2_3 2^* 1 2_2 1_2 2_4 1 2_2 1$$

In view of our assumption on $\lambda_{0}(B)$, by successively applying Lemma \ref{l.F1} (11), (12), we deduce that   
$$B_{-13}\dots B_{15} = 1 2_4 1 2_2 1_2 2_3 2^* 1 2_2 1_2 2_4 1 2_2 1_2 2$$

By Lemma \ref{l.F1} (1) and our assumption on $\lambda_{\pm15}(B)$, we obtain 
$$B_{-13}\dots B_{16} = 1 2_4 1 2_2 1_2 2_3 2^* 1 2_2 1_2 2_4 1 2_2 1_2 2_2$$ 

Finally, by Lemma \ref{l.F1} (13) and our assumption on $\lambda_{0}(B)$, we conclude that 
$$B_{-14}\dots B_{16} = 1_2 2_4 1 2_2 1_2 2_3 2^* 1 2_2 1_2 2_4 1 2_2 1_2 2_2$$
\end{proof}

A careful inspection of the proof of the previous lemma reveals that the following statement holds:

\begin{lemma}\label{l.F3'} Let $B\in\{1,2\}^{\mathbb{Z}}$ such that $B_{-6}\dots B_8 = 2 1_2 2_4 1 2_2 1_2 2_3$ and $\lambda_n(B)<3.1181201786$ for $n\in\{0, -6, -9,11,15\}$. Then, $B_{-14}\dots B_{16}$ equals  
$$1_2 2_4 1 2_2 1_2 2_4 1 2_2 1_2 2_4 1 2_2 1_2 2_2$$

In particular: 
\begin{itemize}
\item either $B_{-15}\dots B_{16} = 1_3 2_4 1 2_2 1_2 2_4 1 2_2 1_2 2_4 1 2_2 1_2 2_2$,
\item or $B_{-15}\dots B_{16} = 2 1_2 2_4 1 2_2 1_2 2_4 1 2_2 1_2 2_4 1 2_2 1_2 2_2$ and the vicinity of $B_{-9}$ is $B_{-15}\dots B_{-1} = 2 1_2 2_4 1 2_2 1_2 2_3$.
\end{itemize} 
\end{lemma}

Using these facts, we can produce another forbidden string for sequences with Markov values in the interval $(c_{\infty}, C_{\infty})$:

\begin{lemma}\label{l.14} Let $A\in\{1,2\}^{\mathbb{Z}}$ such that, for some $n\in\mathbb{Z}$ and $a\in\mathbb{N}$, one has $\lambda_k(A) < 3.1181201786$ for $k-n\in\{-21-6j, -19-6j: j=0,\dots, a\}$, $k=9, 27, 33, 35, 37$, and $k-n\in\{39+6j, 41+6j, 43+6j: j=0,\dots, a\}$.

If $A_{n-15}\dots A_{n+16} = 1_3 2_4 1 2_2 1_2 2_4 1 2_2 1_2 2_4 1 2_2 1_2 2_2$, then
\begin{eqnarray*}
\lambda_n(A)&\geq& [2; 1, 2_2, 1_2, 2_4, 1, 2_2, 1_2, 2_4, 1, 2_2, 1_2, 2_2, 1, 2_4, 1_2, \underbrace{2_3, 1_3, \dots, 2_3, 1_3}_{a+2 \textrm{ times }}, \dots] \\
&+& [0; 2_3, 1_2, 2_2, 1, 2_4, 1,\underbrace{1_3, 2_3, \dots, 1_3, 2_3}_{a+1 \textrm{ times }},\dots]
\end{eqnarray*}

In particular, the subsequence $1_3 2_4 1 2_2 1_2 2_4 1 2_2 1_2 2_4 1 2_2 1_2 2_2$ or its transpose is not contained in a bi-infinite sequence $A\in\{1,2\}^{\mathbb{Z}}$ with $m(A) < C_{\infty}$. 
\end{lemma}

\begin{proof} If $A_{n-15}\dots A_{n+16} = 1_3 2_4 1 2_2 1_2 2_4 1 2_2 1_2 2_4 1 2_2 1_2 2_2$, then 
\begin{eqnarray*}
\lambda_n(A) &=& [0; 2_3, 1_2, 2_2, 1, 2_4, 1_3,\dots] + [2; 1, 2_2, 1_2, 2_4, 1, 2_2, 1_2, 2_2,\dots] \\ 
&\geq& [0; 2_3, 1_2, 2_2, 1, 2_4, 1_4, 2, \dots] + [2; 1, 2_2, 1_2, 2_4, 1, 2_2, 1_2, 2_2,\dots]
\end{eqnarray*}

By Lemma \ref{l.F1} (1) and our assumption on $\lambda_{n-17}(A)$, we have 
\begin{eqnarray*}
\lambda_n(A) &\geq& [0; 2_3, 1_2, 2_2, 1, 2_4, 1_4, 2, \dots] + [2; 1, 2_2, 1_2, 2_4, 1, 2_2, 1_2, 2_3,\dots] \\ 
&\geq& [0; 2_3, 1_2, 2_2, 1, 2_4, 1_4, 2_2, \dots] + [2; 1, 2_2, 1_2, 2_4, 1, 2_2, 1_2, 2_3,\dots] \\ 
&\geq& [0; 2_3, 1_2, 2_2, 1, 2_4, 1_4, 2_3, 1, \dots] + [2; 1, 2_2, 1_2, 2_4, 1, 2_2, 1_2, 2_3,\dots]
\end{eqnarray*}

By Lemma \ref{l.F1} (1), (2) and our assumption on $\lambda_{n-21}(A)$, $\lambda_{n-19}(A)$, we get  
\begin{eqnarray*}
\lambda_n(A) &\geq& [0; 2_3, 1_2, 2_2, 1, 2_4, 1_4, 2_3, 1, \dots] + [2; 1, 2_2, 1_2, 2_4, 1, 2_2, 1_2, 2_3,\dots] \\ 
&\geq& [0; 2_3, 1_2, 2_2, 1, 2_4, 1_4, 2_3, 1_2, \dots] + [2; 1, 2_2, 1_2, 2_4, 1, 2_2, 1_2, 2_3,\dots] \\ 
&\geq& [0; 2_3, 1_2, 2_2, 1, 2_4, 1_4, 2_3, 1_3, 2, \dots] + [2; 1, 2_2, 1_2, 2_4, 1, 2_2, 1_2, 2_3,\dots]
\end{eqnarray*}

By Lemma \ref{l.F1} (1) and our assumption on $\lambda_{n-23}(A)$, we obtain 
\begin{eqnarray*}
\lambda_n(A) &\geq& [0; 2_3, 1_2, 2_2, 1, 2_4, 1_4, 2_3, 1_3, 2, \dots] + [2; 1, 2_2, 1_2, 2_4, 1, 2_2, 1_2, 2_3,\dots] \\ 
&\geq& [0; 2_3, 1_2, 2_2, 1, 2_4, 1_4, 2_3, 1_3, 2_2, \dots] + [2; 1, 2_2, 1_2, 2_4, 1, 2_2, 1_2, 2_3,\dots] \\ 
&\geq& [0; 2_3, 1_2, 2_2, 1, 2_4, 1_4, 2_3, 1_3, 2_3, \dots] + [2; 1, 2_2, 1_2, 2_4, 1, 2_2, 1_2, 2_3,\dots]
\end{eqnarray*}

By recursively applying the previous arguments at the positions $20+6j$, $1\leq j\leq a$, we see that our assumptions on $\lambda_{n-19-6j}(A)$ and $\lambda_{n-21-6j}(A)$ for $1\leq j\leq a$ together with Lemma \ref{l.F1} (1), (2) imply that 
\begin{eqnarray*}
\lambda_n(A)&\geq& [0; 2_3, 1_2, 2_2, 1, 2_4, 1,\underbrace{1_3, 2_3, \dots, 1_3, 2_3}_{a+1 \textrm{ times }},\dots] \\
&+& [2; 1, 2_2, 1_2, 2_4, 1, 2_2, 1_2, 2_3,\dots]
\end{eqnarray*}

On the other hand, the fact that $A_{n}\dots A_{n} = 2 1_2 2_4 1 2_2 1_2 2_3$ and our assumption on $\lambda_9(A)$ allow to apply Lemma \ref{l.F3'} to get  
\begin{eqnarray*}
\lambda_n(A)&\geq& [0; 2_3, 1_2, 2_2, 1, 2_4, 1,\underbrace{1_3, 2_3, \dots, 1_3, 2_3}_{a+1 \textrm{ times }},\dots] \\
&+& [2; 1, 2_2, 1_2, 2_4, 1, 2_2, 1_2, 2_4, 1, 2_2, 1_2, 2_2, \dots]
\end{eqnarray*}

By Lemma \ref{l.F1} (1), (2), (4) and our assumption on $\lambda_{27}(A)$, we have 
\begin{eqnarray*}
\lambda_n(A)&\geq& [0; 2_3, 1_2, 2_2, 1, 2_4, 1,\underbrace{1_3, 2_3, \dots, 1_3, 2_3}_{a+1 \textrm{ times }},\dots] \\
&+& [2; 1, 2_2, 1_2, 2_4, 1, 2_2, 1_2, 2_4, 1, 2_2, 1_2, 2_2, 1, 2_4, 1_2, 2, \dots]
\end{eqnarray*}

By Lemma \ref{l.F1} (1) and our assumption on $\lambda_{33}(A)$, we have 
\begin{eqnarray*}
\lambda_n(A)&\geq& [0; 2_3, 1_2, 2_2, 1, 2_4, 1,\underbrace{1_3, 2_3, \dots, 1_3, 2_3}_{a+1 \textrm{ times }},\dots] \\
&+& [2; 1, 2_2, 1_2, 2_4, 1, 2_2, 1_2, 2_4, 1, 2_2, 1_2, 2_2, 1, 2_4, 1_2, 2_3, 1, \dots]
\end{eqnarray*}

By Lemma \ref{l.F1} (1) and our assumption on $\lambda_{37}(A)$, and by Lemma \ref{l.F1} (2) and our assumption on $\lambda_{35}(A)$, we have 
\begin{eqnarray*}
\lambda_n(A)&\geq& [0; 2_3, 1_2, 2_2, 1, 2_4, 1,\underbrace{1_3, 2_3, \dots, 1_3, 2_3}_{a+1 \textrm{ times }},\dots] \\
&+& [2; 1, 2_2, 1_2, 2_4, 1, 2_2, 1_2, 2_4, 1, 2_2, 1_2, 2_2, 1, 2_4, 1_2, 2_3, 1_3, \dots]
\end{eqnarray*}

By recursively applying Lemma \ref{l.F1} (1) at the positions $39+6j$, $39+4(j+1)$, $0\leq j\leq a$, and Lemma \ref{l.F1} (2) at the positions $39+2(j+1)$, $0\leq j\leq a$, we see that our assumptions on $\lambda_{n+39+6j}(A)$, $\lambda_{n+43+6j}(A)$ and $\lambda_{n+41+6j}(A)$ for $0\leq j\leq a$ imply that 

\begin{eqnarray*}
\lambda_n(A)&\geq& [0; 2_3, 1_2, 2_2, 1, 2_4, 1,\underbrace{1_3, 2_3, \dots, 1_3, 2_3}_{a+1 \textrm{ times }},\dots] \\
&+& [2; 1, 2_2, 1_2, 2_4, 1, 2_2, 1_2, 2_4, 1, 2_2, 1_2, 2_2, 1, 2_4, 1_2, 2_3, 1_3, \underbrace{2_3, 1_3, \dots, 2_3, 1_3}_{a+1 \textrm{ times }}, \dots]
\end{eqnarray*}

Finally, assume that $A\in\{1,2\}^{\mathbb{Z}}$ is a bi-infinite sequence with $m(A)<C_{\infty}$ containing $1_3 2_4 1 2_2 1_2 2_4 1 2_2 1_2 2_4 1 2_2 1_2 2_2$ or its transpose, say $A_{l-15}\dots A_{l+16}$ or $(A_{l-16}\dots A_{l+15})^T$ equals $1_3 2_4 1 2_2 1_2 2_4 1 2_2 1_2 2_4 1 2_2 1_2 2_2$ for some $l\in\mathbb{Z}$. Our discussion above would then imply that
\begin{eqnarray*}
C_{\infty} &>& m(A) \geq \lambda_l(A) \\ &\geq& [2; 1, 2_2, 1_2, 2_4, 1, 2_2, 1_2, 2_4, 1, 2_2, 1_2, 2_2, 1, 2_4, 1_2, \overline{2_3, 1_3}]
+ [0; 2_3, 1_2, 2_2, 1, 2_4, 1,\overline{1_3, 2_3}] \\
&:=& C_{\infty},
\end{eqnarray*}
a contradiction. This proves the lemma.
\end{proof}

At this point, we are ready to characterize the sequences in $\{1,2\}^{\mathbb{Z}}$ giving rise to a Markov value in $(c_{\infty}, C_{\infty})$: 

\begin{proposition}\label{p.Cantors-covering-M-L-piece} Let $m\in M\cap (c_{\infty}, C_{\infty})$. Then, $m=m(B)=\lambda_0(B)$ for a sequence $B\in\{1,2\}^{\mathbb{Z}}$ with the following properties:
\begin{itemize}
\item $\dots B_{-14}\dots B_0 B_1\dots B_{16} = \overline{1 2_2 1_2 2_4}1 2_2 1_2 2_4 1 2_2 1_2 2_2$;
\item there exists $N\geq 17$ such that $B_{N} B_{N+1}\dots$ is a word on $1$ and $2$ satisfying:
\begin{itemize}
\item it does not contain the subwords (1) to (13) and their transposes, and $2 1_2 2_4 1 2_2 1_2 2_3$,
\item if it contains the subword $2_3 1_2 2_2 1 2_4 1_2 2 = B_{n-8}\dots B_{n+6}$, then
$$B_{n}\dots B_{n+9} \dots = \overline{2_4 1 2_2 1_2}$$
\end{itemize}
\end{itemize}
\end{proposition}

\begin{proof} Let $B\in\{1,2\}^{\mathbb{Z}}$ be a bi-infinite sequence such that $\lambda_0(B) = m(B) = m$. Since $3.118117 < c_{\infty} < m < C_{\infty} < 3.1181201786$, Lemma \ref{l.F4} says that
$B_{-14}\dots B_{16}$ or $(B_{-16}\dots B_{14})^T$ equals to $1_2 2_4 1 2_2 1_2 2_4 1 2_2 1_2 2_4 1 2_2 1_2 2_2$. 

Thus, by reversing $B$ if necessary, we obtain a bi-infinite sequence $B\in\{1,2\}^{\mathbb{Z}}$ such that $m=m(B)=\lambda_0(B)$ and $B_{-14}\dots B_{16} = 1_2 2_4 1 2_2 1_2 2_4 1 2_2 1_2 2_4 1 2_2 1_2 2_2$.

Because $m(B)=m<C_{\infty}$, we know from Lemma \ref{l.14} that $B$ does not contain the word $1_3 2_4 1 2_2 1_2 2_4 1 2_2 1_2 2_4 1 2_2 1_2 2_2$, and, thus, we can successively apply Lemma \ref{l.F3'} at the positions $-9k$, $k\in\mathbb{N}$, to get that
$$\dots B_{-14}\dots B_0 B_1\dots B_{16} = \overline{1 2_2 1_2 2_4}1 2_2 1_2 2_4 1 2_2 1_2 2_2$$
Moreover, Lemma \ref{l.F1} implies that the word $B_{17}\dots$ does not contain the subwords (1) to (13) and their transposes.

Furthermore, the subword $2 1_2 2_4 1 2_2 1_2 2_3$ can not appear in $B_n\dots$ for all $n\geq 17$. Indeed, if this happens, since $m(B)=m<C_{\infty}$, it would follow from Lemma \ref{l.14} that $B$ does not contain the subsequence $1_3 2_4 1 2_2 1_2 2_4 1 2_2 1_2 2_4 1 2_2 1_2 2_2$ and, hence, one could repeatedly apply Lemma \ref{l.F3'} to deduce that $B=\overline{1 2_2 1_2 2_4}$, a contradiction because this would mean that $c_{\infty}<m=m(B)=m(\overline{1 2_2 1_2 2_4})=c_{\infty}$.

In summary, we showed that there exists $N\geq 17$ such that the word $B_{N}\dots$ does not contain the subwords (1) to (13) and their transposes, and $2 1_2 2_4 1 2_2 1_2 2_3$.

Finally, if the word $B_{17}\dots$ contains the subword $2_3 1_2 2_2 1 2_4 1_2 2 = B_{n-8}\dots B_{n+6}$, since $B$ does not contain the transpose of $1_3 2_4 1 2_2 1_2 2_4 1 2_2 1_2 2_4 1 2_2 1_2 2_2$ (thanks to Lemma \ref{l.14} and the fact that $m(B)=m<C_{\infty}$), then one can apply Lemma \ref{l.F3'} at the positions $n+9k$ for all $k\in\mathbb{N}$ to get that
$$B_{n}\dots B_{n+9} \dots = \overline{2_4 1 2_2 1_2}$$

This completes the argument. 
\end{proof}

\begin{remark} We use Proposition \ref{p.Cantors-covering-M-L-piece} to detect new numbers in $M\setminus L$: see Section \ref{a.new-numbers}.
\end{remark}

A variant of the argument used in the proof of Proposition \ref{p.Cantors-covering-M-L-piece} yields the  following result (essentially due to Berstein \cite{Be73}): 

\begin{proposition}\label{p.Berstein1} $L\cap (c_{\infty}, C_{\infty}) = \emptyset$.
\end{proposition}

\begin{proof} Suppose that $\alpha\in L\cap (c_{\infty}, C_{\infty})$ and fix $B\in\{1,2\}^{\mathbb{Z}}$ with $\ell(B) = \alpha$. Since $3.118117 < c_{\infty} < \alpha < C_{\infty} < 3.1181201786$, we can choose $N\in\mathbb{N}$ such that 
$$\lambda_n(B)< 3.1181201786$$
for all $|n|\geq N$, and we can fix a monotone sequence $\{n_k\}_{k\in\mathbb{N}}$ such that $|n_k|\geq N$ and $\lambda_{n_k}(B) > 3.118117$ for all $k\in\mathbb{Z}$. Moreover, by reversing $B$ if necessary, we can assume that $n_k\to+\infty$ as $k\to\infty$ and $\limsup\limits_{n\to+\infty}\lambda_n(B) = \alpha$.

We affirm that $1_3 2_4 1 2_2 1_2 2_4 1 2_2 1_2 2_4 1 2_2 1_2 2_2$ or its transpose can not be contained in $B_n B_{n+1}\dots$ for all $n\geq N$: otherwise, we would have a sequence  $m_k\to+\infty$ as $k\to\infty$ such that $B_{m_k-15}\dots B_{m_k+16}$ or $(B_{m_k-16}\dots B_{m_k+15})^T$ equals $1_3 2_4 1 2_2 1_2 2_4 1 2_2 1_2 2_4 1 2_2 1_2 2_2$  for all $k\in\mathbb{N}$; hence, by Lemma \ref{l.14}, the fact that $\lambda_n(B)<3.1181201786$ for all $n\geq N$ would imply that
\begin{eqnarray*}
\lambda_{m_k}(B)&\geq& [2; 1, 2_2, 1_2, 2_4, 1, 2_2, 1_2, 2_4, 1, 2_2, 1_2, 2_2, 1, 2_4, 1_2, \underbrace{2_3, 1_3, \dots, 2_3, 1_3}_{a_k+2 \textrm{ times }}, \dots] \\
&+& [0; 2_3, 1_2, 2_2, 1, 2_4, 1,\underbrace{1_3, 2_3, \dots, 1_3, 2_3}_{a_k+1 \textrm{ times }},\dots]
\end{eqnarray*}
where $a_k=\lfloor\frac{m_k-43-N}{6}\rfloor$; on the other hand, since $a_k\to\infty$ as $k\to\infty$, it would follow that $C_{\infty} > \alpha\geq \limsup\limits_{k\to\infty} \lambda_{m_k}(B)\geq C_{\infty}$, 
a contradiction. 

Thus, the discussion of the previous paragraph allows us to select $R\geq N$ such that $B_R B_{R+1}\dots$ does not contain $1_3 2_4 1 2_2 1_2 2_4 1 2_2 1_2 2_4 1 2_2 1_2 2_2$ or its transpose. 

Note that, by Lemma \ref{l.F4}, our choices of $N\in\mathbb{N}$ and $\{n_k\}_{k\in\mathbb{N}}$ imply that 
\begin{itemize}
\item either $B_{n_k-14}\dots B_{n_k+16}=1_2 2_4 1 2_2 1_2 2_4 1 2_2 1_2 2_4 1 2_2 1_2 2_2$
\item or $(B_{n_k-16}\dots B_{n_k+14})^T = 1_2 2_4 1 2_2 1_2 2_4 1 2_2 1_2 2_4 1 2_2 1_2 2_2$.
\end{itemize}
for each $k\in\mathbb{N}$ with $n_k\geq N+15$. 

If the first possibility occurs for all $n_k > R+15$, then the facts that $\lambda_n(B)<3.1181201786$ for all $n\geq N$ and the sequence $B_RB_{R+1}\dots$ does not contain $1_3 2_4 1 2_2 1_2 2_4 1 2_2 1_2 2_4 1 2_2 1_2 2_2$ allow to apply $d_k:=\lfloor\frac{n_k-6-R}{9}\rfloor$ times Lemma \ref{l.F3'} at the positions $n_k-9(j-1)$, $j=1,\dots, d_k$ to deduce that the sequence $B$ has the form
$$\dots B_{n_k-9d_k}\dots B_{n_k}\dots B_{n_k+16}\dots = \dots \underbrace{212_21_22_3,\dots, 212_21_22_3}_{d_k \textrm{ times}} 2 1 2_2 1_2 2_4 1 2_2 1_2 2_2\dots$$
Because $R-15\leq n_k-9d_k\leq R+16$ and $n_k\to+\infty$, we get that $B$ has the form $\dots\overline{212_21_22_3}$.

If the second possibility occurs for some $k_0\in\mathbb{N}$ with $n_{k_0}\geq R+15$, then the facts that $\lambda_n(B)<B_{\infty}<\alpha_{\infty}+10^{-6}$ for all $n\geq N$ and the sequence $B_RB_{R+1}\dots$ does not contain the subsequence $(1_3 2_4 1 2_2 1_2 2_4 1 2_2 1_2 2_4 1 2_2 1_2 2_2)^T$ allow to apply Lemma \ref{l.F3'} at the positions $n_k+9a$, $a\in\mathbb{N}$, to deduce that the sequence $B$ has the form
$$\dots B_{n_{k_0}} B_{n_{k_0}+1}\dots B_{n_k+10}\dots = \dots \overline{2_4 1_2 2_2 1}$$

In any case, each possibility above would imply that
$$c_{\infty} < \alpha=\limsup\limits_{n\to+\infty}\lambda_n(B) = \ell(\overline{2_4 1_2 2_2 1})=c_{\infty},$$
a contradiction.

In summary, the existence of $\alpha\in L\cap (c_{\infty}, C_{\infty})$ would lead to a contradiction in any scenario. This proves the proposition.
\end{proof}

\begin{remark}\label{r.Berstein} As it was first observed in Theorem 1, pages 47 to 49 of Berstein's article \cite{Be73}, one can \emph{improve} Proposition \ref{p.Berstein1} by showing that $(c_{\infty}, C_{\infty})$ is the \emph{largest} interval disjoint from $L$ containing $\sigma$.

Actually, it is not difficult to show this refinement of Proposition \ref{p.Berstein1}: indeed, since this proposition ensures that $L\cap(c_{\infty}, C_{\infty}) = \emptyset$, and we have $c_{\infty}=\ell(\overline{2_4 1_2 2_2 1})\in L$, it suffices to prove that $C_{\infty}\in L$. For the sake of exposition (and to correct some mistakes in \cite{Be73}), we show this fact in Appendix \ref{a.Berstein} below.
\end{remark}

\section{Proof of Theorem \ref{t.M-L-piece-HD}}\label{s.proof}
The description of $(M\setminus L)\cap  (c_{\infty}, C_{\infty}) = M\cap  (c_{\infty}, C_{\infty})$ provided by Propositions \ref{p.Cantors-covering-M-L-piece} and \ref{p.Berstein1} allows us to compare this piece of $M\setminus L$ with the Cantor set
\begin{equation*}
Y:=\{[0;\gamma]:\gamma\in\{1,2\}^{\mathbb{N}} \textrm{ not containing the subwords in } P\}
\end{equation*}
where $P$ is the finite set of 27 words consisting of the words (1) to (13) in Lemma \ref{l.F1} above and their transposes, and the words $2 1_2 2_4 1 2_2 1_2 2_3$ and its transpose. 

\begin{proposition}\label{p.M-L-piece-HD1} $(M\setminus L)\cap  (3.118120178159, 3.118120178173)$ contains the set
$$\{[\overline{2_4, 1_2, 2_2, 1}] + [0; 1, 2_2, 1_2, 2_4, 1, 2_2, 1_2, 2_2, 1_2, \gamma]: 1_2 2_2 1_2 \gamma\in \{1,2\}^{\mathbb{N}} \textrm{ does not contain the subwords in } P\}$$
\end{proposition}

\begin{proof} Consider $B=\overline{1 2_2 1_2 2_4}; 1 2_2 1_2 2_4 1 2_2 1_2 2_2 1_2\gamma$ where $1_2 2_2 1_2 \gamma\in \{1, 2\}^{\mathbb{N}}$ does not contain subwords in $P$ and $;$ serves to indicate the zeroth position.

Note that 
$$\lambda_0(B)\leq [\overline{2_4, 1_2, 2_2, 1}] + [0; 1, 2_2, 1_2, 2_4, 1, 2_2, 1_2, 2_2, 1, \overline{2, 1}] < 3.118120178173$$
and
$$\lambda_0(B)\geq [\overline{2_4, 1_2, 2_2, 1}] + [0; 1, 2_2, 1_2, 2_4, 1, 2_2, 1_2, 2_2, 1, \overline{1, 2}] >3.118120178159,$$
and Lemma \ref{l.F2} (15), (16), (19) imply that $\lambda_n(B)<3.118117$ for all positions $n\leq 18$ except possibly for $n=-9k$ with $k\geq 0$.

On the other hand,
\begin{eqnarray*}\lambda_{-9k}(B) &=& [\overline{2_4, 1_2, 2_2, 1}] + [0; \underbrace{1, 2_2, 1_2, 2_4, \dots, 1, 2_2, 1_2, 2_4}_{k \textrm{ times }}, 1, 2_2, 1_2, 2_4, 1, 2_2, 1_2, 2_2, 1_2, \gamma]  \\ &<& [\overline{2_4, 1_2, 2_2, 1}] + [0; 1, 2_2, 1_2, 2_4, 1, 2_2, 1_2, 2_2, 1_2, \gamma] = \lambda_0(B),
\end{eqnarray*}
for all $k\geq 1$.

Furthermore, since $1_2 2_2 1_2 \gamma$ does not contain subwords in $P$, it follows from (the proof of) Lemma \ref{l.F4} that $\lambda_n(B)<3.118117$ for all $n\geq 19$.

This shows that $m(B)=\lambda_0(B)= [\overline{2_4, 1_2, 2_2, 1}] + [0; 1, 2_2, 1_2, 2_4, 1, 2_2, 1_2, 2_2, 1_2, \gamma]$ belongs to $(M\setminus L)\cap  (3.118120178159, 3.118120178173)$.
\end{proof}

\begin{proposition}\label{p.M-L-piece-HD2} $(M\setminus L)\cap  (c_{\infty}, C_{\infty})$ is contained in the union of
$$\mathcal{C} = \{[0; \overline{2_3, 1_2, 2_2, 1, 2}] + [2; 1, 2_2, 1_2, 2_4, 1, 2_2, 1_2, 2_2, \theta,\overline{2_4 1 2_2 1_2}]: \theta \textrm{ is a finite word in } 1 \textrm{ and } 2\}$$
and the sets
$$\mathcal{D}(\delta) = \{0; \overline{2_3, 1_2, 2_2, 1, 2}] + [2; 1, 2_2, 1_2, 2_4, 1, 2_2, 1_2, 2_2, \delta,\gamma]: \textrm{ no subword of }\gamma\in \{1,2\}^{\mathbb{N}} \textrm{ belongs to } P\},$$
where $\delta$ is a finite word in $1$ and $2$.
\end{proposition}

\begin{proof} By Proposition \ref{p.Cantors-covering-M-L-piece}, if $m\in (M\setminus L)\cap (c_{\infty}, C_{\infty})$, then $m=m(B)=\lambda_0(B)$ with
$$B=\overline{21 2_2 1_2 2_3}2^*1 2_2 1_2 2_4 1 2_2 1_2 2_2\delta\gamma$$
where the asterisk indicates the zeroth position, $\delta$ is a finite word in $1$ and $2$, and the infinite word $\gamma$ satisfies:
\begin{itemize}
\item $\gamma$ does not contain the subwords (1) to (13) and their transposes, and $2 1_2 2_4 1 2_2 1_2 2_3$, 
\item if $\gamma$ contains the subword $2_3 1_2 2_2 1 2_4 1_2 2$, then
$\gamma = \mu\overline{2_4 1 2_2 1_2}$ with $\mu$ a finite word in $1$ and $2$.
\end{itemize}

Hence:
\begin{itemize}
\item if $\gamma$ contains $2_3 1_2 2_2 1 2_4 1_2 2$, then
$$m(B)=\lambda_0(B)=[0; \overline{2_3, 1_2, 2_2, 1, 2}] + [2; 1, 2_2, 1_2, 2_4, 1, 2_2, 1_2, 2_2, \delta,\mu,\overline{2_4 1 2_2 1_2}]$$
where $\theta=\delta\mu$ is a finite word in $1$ and $2$, i.e., $m(B)\in\mathcal{C}$;
\item otherwise,
$$m(B) = \lambda_0(B)=[0; \overline{2_3, 1_2, 2_2, 1, 2}] + [2; 1, 2_2, 1_2, 2_4, 1, 2_2, 1_2, 2_2, \delta,\gamma]$$
where $\gamma$ does not contain the subwords (1) to (13) and their transposes, and $2 1_2 2_4 1 2_2 1_2 2_3$ and its transpose $2_3 1_2 2_2 1 2_4 1_2 2$, i.e., $m(B)\in\mathcal{D}(\delta)$.
\end{itemize}
This completes the argument.
\end{proof}

At this stage, Theorem \ref{t.M-L-piece-HD} follows directly from Propositions \ref{p.M-L-piece-HD1} and \ref{p.M-L-piece-HD2}: on one hand, by Proposition \ref{p.M-L-piece-HD1}, $(M\setminus L)\cap  (c_{\infty}, C_{\infty})$ contains a set diffeomorphic to $Y$ and, hence,
$$HD((M\setminus L)\cap  (c_{\infty}, C_{\infty})) \geq HD(Y)$$
On the other hand, by Proposition \ref{p.M-L-piece-HD2}, $(M\setminus L)\cap   (c_{\infty}, C_{\infty})$ is contained in
$$\mathcal{C}\cup\bigcup\limits_{n\in\mathbb{N}} \left(\bigcup\limits_{\delta\in\{1,2\}^n}\mathcal{D}(\delta)\right)$$
Since $\mathcal{C}$ is a countable set and $\{\mathcal{D}(\delta):\delta\in\{1,2\}^n, n\in\mathbb{N}\}$ is a countable family of subsets diffeomorphic to $Y$, it follows that
$$HD((M\setminus L)\cap (c_{\infty}, C_{\infty})) \leq HD(Y)$$

This proves Theorem \ref{t.M-L-piece-HD}.

\section{Lower bounds on $HD((M\setminus L)\cap (c_{\infty}, C_{\infty}))$}\label{a.PT}

Note that the definition of $Y$ in \eqref{e.Cantor-Y} implies that $Y$ contains the Gauss-Cantor set $K(\{1_2, 2_2\})$. Thus, Theorem \ref{t.M-L-piece-HD} implies that 
$$HD((M\setminus L)\cap (c_{\infty}, C_{\infty})) = HD(Y) \geq HD(K(\{1_2, 2_2\}))$$

In this section, we complete the proof of Theorem \ref{t.A} by showing that:

\begin{proposition}\label{p.HD(M-L)>0} One has $0.2628 < HD(K(\{1_2, 2_2\})) < 0.2646$.
\end{proposition}

The convex hull of $K(\{1_2,2_2\})$ is the interval $I$ with extremities $[0;\overline{1}]$ and $[0;\overline{2}]$. The images $I_{11}:=\phi_{11}(I)$ and $I_{22} := \phi_{22}(I)$ of $I$ under the inverse branches
$$\phi_{11}(x):= \frac{1}{1+\frac{1}{1+\frac{1}{x}}} \quad \textrm{and} \quad \phi_{22}(x) := \frac{1}{2+\frac{1}{2+\frac{1}{x}}}$$
of the first two iterates of the Gauss map $G(x):=\{1/x\}$ provide the first step of the construction of the Cantor set $K(\{1_2,2_2\})$. 

The collection $\mathcal{R}^n$ of intervals of the $n$th step of the construction of $K(\{1_2,2_2\})$ is 
$$\mathcal{R}^n:=\{\phi_{x_1}\circ\dots\circ\phi_{x_n}(I): (x_1,\dots, x_n)\in\{11, 22\}^n\}$$

Given $R\in\mathcal{R}^n$, let
$$\lambda_{n,R}:=\inf\limits_{x\in R}|(\Psi^n)'(x)|, \quad \Lambda_{n,R}:=\sup\limits_{y\in R}|(\Psi^n)'(y)|,$$
and define $\alpha_n\in [0,1]$, $\beta_n\in [0,1]$ by
$$\sum\limits_{R\in\mathcal{R}^n} \left(\frac{1}{\Lambda_{n,R}}\right)^{\alpha_n} = 1 = \sum\limits_{R\in\mathcal{R}^n} \left(\frac{1}{\lambda_{n,R}}\right)^{\beta_n}$$

It is shown in \cite[pp. 69--70]{PT} that $\alpha_n\leq HD(K(\{1_2,2_2\}))\leq\beta_n$ for all $n\in\mathbb{N}$.

Therefore, we can estimate on $K(\{1_2, 2_2\})$ by computing $\alpha_n$ and $\beta_n$ for some particular values of $n\in\mathbb{N}$.

In this direction, we observe that
$$\lambda_{n,R} = \min\left\{\prod\limits_{i=1}^{n}\left(\frac{1}{[0;x_i,\dots, x_n, \overline{1}]}\right)^2, \prod\limits_{i=1}^{n}\left(\frac{1}{[0;x_i,\dots, x_n, \overline{2}]}\right)^2 \right\}$$
and
$$\Lambda_{n,R} = \max\left\{\prod\limits_{i=1}^{n}\left(\frac{1}{[0;x_i,\dots, x_n, \overline{1}]}\right)^2, \prod\limits_{i=1}^{n}\left(\frac{1}{[0;x_i,\dots, x_n, \overline{2}]}\right)^2 \right\}$$ 
for $R=\psi_{x_1}\circ\dots\circ\psi_{x_n}(I)\in\mathcal{R}^n$ associated to a string $(x_1,\dots, x_n)\in\{11, 22\}^n$.

Thus, $\alpha_n$ and $\beta_n$ are the solutions of
$$\sum\limits_{(x_1,\dots,x_n)\in\{11,22\}^n}\left(\min\{[0;x_i,\dots, x_n, \overline{1}], [0;x_i,\dots, x_n, \overline{2}]\}\right)^{2\alpha_n}=1$$
and
$$\sum\limits_{(x_1,\dots,x_n)\in\{11,22\}^n}\left(\max\{[0;x_i,\dots, x_n, \overline{1}], [0;x_i,\dots, x_n, \overline{2}]\}\right)^{2\beta_n}=1$$

A quick computer search for the values of $\alpha_{12}$ and $\beta_{12}$ reveals that
$$\alpha_{12} = 0.2628... \quad \textrm{and} \quad \beta_{12} = 0.2645...$$

In particular, $0.2628<\alpha_{12}\leq HD(K(\{1_2,2_2\}))\leq \beta_{12}<0.2646$, so that the proof of Proposition \ref{p.HD(M-L)>0} and, \emph{a fortiori}, Theorem \ref{t.A} is now complete.

\section{The smallest known number in $M\setminus L$}\label{a.new-numbers}

Consider the sequence $\rho\in \{1,2\}^{\mathbb{Z}}$ given by
$$\rho:=\overline{1 2_2 1_2 2_4};1 2_2 1_2 2_4 1 2_2 1_2 2_2 1_2 \overline{2_3 1_3}$$
where $;$ serves to indicate the zeroth position.

In this section, we show\footnote{To the best of our knowledge, the previous smallest known elements of $M\setminus L$ appearing in the literature were the elements of the countable subset $\mathcal{F}$ described by Freiman in the proof of Theorem 3 at page 200 of \cite{Fr68}. Concretely, $\mathcal{F}$ is the set of Markov values of the sequences $S(w):=\overline{1 2_2 1_2 2_4};1 2_2 1_2 2_4 1 2_2 1_2 2_2 1 w 1  \overline{2_2 1_2 2_2 1 2_2}$ where $w$ is any finite word in $1$ and $2$ such that $m(S(w))=\lambda_0(S(w))$. In this setting, the constant $\sigma = 3.11812017815993\dots\in\mathcal{F}$ explicitly mentioned by Freiman at page 195 of \cite{Fr68} corresponds to the empty word, i.e., $\sigma = \lambda_0(S(\emptyset))$. Note that $\sigma>f$ and, more generally, our proof of Lemma \ref{l.gamma-bound} below says that $x>f$ for all $x\in\mathcal{F}$. Nevertheless, one can easily choose finite words $w$ in order to check that $\inf\mathcal{F}=f$.} that
\begin{eqnarray*}
f&:=&\lambda_0(\rho)= [\overline{2_4, 1_2, 2_2, 1}]+[0;1, 2_2, 1_2, 2_4, 1, 2_2, 1_2, 2_2, 1_2, \overline{2_3, 1_3}] \\ 
&=& 3.11812017815984\dots
\end{eqnarray*}
is the smallest element of $(M\setminus L)\cap (c_{\infty}, C_{\infty})$.

We begin by proving that $f\in M$:

\begin{lemma}\label{l.gamma-in-M} One has $f=\lambda_0(\rho)=m(\rho)\in M$.
\end{lemma}

\begin{proof}
From Lemma \ref{l.F2} (15), (16), (19), one has $\lambda_j(\rho)<3.118117$ for all $j\in\mathbb{Z}$ except possibly for $j=-9k$, $k\geq 0$. 

On the other hand, we have
\begin{eqnarray*}
 \lambda_{-9k}(\rho) &=& [\overline{2_4, 1_2, 2_2, 1}]+[0;\underbrace{1, 2_2, 1_2, 2_4,\dots, 1, 2_2, 1_2, 2_4}_{k+1 \textrm{ times }}, 1, 2_2, 1_2, 2_2, 1_2, \overline{2_3, 1_3}] \\
   &<& [\overline{2_4, 1_2, 2_2, 1}]+[0;1, 2_2, 1_2, 2_4, 1, 2_2, 1_2, 2_2, 1_2, \overline{2_3, 1_3}] = \lambda_0(\rho)
\end{eqnarray*}
for each $k\geq 1$.

In other terms, we showed that $\lambda_j(\rho)<\lambda_0(\rho)$ for all $j\neq 0$, and, \emph{a fortiori}, $f=\lambda_0(\rho)=m(\rho)\in M$.
\end{proof}

Let us now show that $m\geq f$ for all $m\in M\cap(c_{\infty}, C_{\infty})$:
\begin{lemma}\label{l.gamma-bound}
  If $m\in M\cap(c_{\infty},C_{\infty})$, then $m\geq f$.
\end{lemma}

\begin{proof} By Proposition \ref{p.Cantors-covering-M-L-piece}, any $m\in M\cap(c_{\infty},C_{\infty})$ has the form:
$$m=\lambda_0(B) = m(B) = [\overline{2_4, 1_2, 2_2, 1}]+[0;1, 2_2, 1_2, 2_4, 1, 2_2, 1_2, 2_2,\dots]$$

We claim there exists a smallest integer $k_0\in\mathbb{N}$ such that $B_{17+9k_0}\neq 2$: otherwise, since $m(B)<C_{\infty}<3.1181201786$, we could recursively apply Lemma \ref{l.F3'} at the positions $n=17+9k$ to deduce that $B=\overline{1, 2_2, 1_2, 2_4}$, and, hence $c_{\infty}=m(\overline{2_4, 1_2, 2_2, 1})=m(B)$, a contradiction.

By definition of $k_0$, 
\begin{eqnarray*}
m(B)&\geq& \lambda_{9k_0}(B)=[\overline{2_4, 1_2, 2_2, 1}]+[0;1, 2_2, 1_2, 2_4, 1, 2_2, 1_2, 2_2,1,\dots] \\ 
&\geq& [\overline{2_4, 1_2, 2_2, 1}]+[0;1, 2_2, 1_2, 2_4, 1, 2_2, 1_2, 2_2, 1_2, 2,\dots]
\end{eqnarray*}

By Lemma \ref{l.F1} (1), we deduce that 
\begin{eqnarray*}
m(B)&\geq& [\overline{2_4, 1_2, 2_2, 1}]+[0;1, 2_2, 1_2, 2_4, 1, 2_2, 1_2, 2_2, 1_2, 2,\dots] \\ 
&\geq& [\overline{2_4, 1_2, 2_2, 1}]+[0;1, 2_2, 1_2, 2_4, 1, 2_2, 1_2, 2_2, 1_2, 2_2,\dots] \\ 
&\geq& [\overline{2_4, 1_2, 2_2, 1}]+[0;1, 2_2, 1_2, 2_4, 1, 2_2, 1_2, 2_2, 1_2, 2_3,1\dots]
\end{eqnarray*}

By Lemma \ref{l.F1} (1), (2), we obtain 
\begin{eqnarray*}
m(B)&\geq& [\overline{2_4, 1_2, 2_2, 1}]+[0;1, 2_2, 1_2, 2_4, 1, 2_2, 1_2, 2_2, 1_2, 2_3,1\dots] \\ 
&\geq& [\overline{2_4, 1_2, 2_2, 1}]+[0;1, 2_2, 1_2, 2_4, 1, 2_2, 1_2, 2_2, 1_2, 2_3,1_3\dots]
\end{eqnarray*}

At this point, we can repeat the recursive argument in the proof of Lemma \ref{l.14} to conclude that 
\begin{eqnarray*}
m(B)&\geq& [\overline{2_4, 1_2, 2_2, 1}]+[0;1, 2_2, 1_2, 2_4, 1, 2_2, 1_2, 2_2, 1_2, 2_3,1_3\dots] \\ 
&\geq& [\overline{2_4, 1_2, 2_2, 1}]+[0;1, 2_2, 1_2, 2_4, 1, 2_2, 1_2, 2_2, 1_2, \overline{2_3,1_3}] = f
\end{eqnarray*}

This completes the proof of the lemma. 
\end{proof}

\appendix

\section{Berstein's interval around $\sigma$}\label{a.Berstein}

In this appendix, we prove that $(c_{\infty}, C_{\infty})$ is the largest interval disjoint from $L$ containing $\sigma$. As it was pointed out in Remark \ref{r.Berstein}, our task is reduced to prove that:

\begin{lemma}
$C_{\infty}\in L$.
\end{lemma}

\begin{proof}
Our task is to exhibit a sequence $(P_a)_{a\in\mathbb{N}}$ of finite words in $1$ and $2$ such that 
  $$\lim\limits_{a\to\infty} \ell(\overline{P_a})=C_{\infty}$$
  
We claim that   
  $$P_a:=\underbrace{2_3 1_3\dots 2_3 1_3}_{a \textrm{ times }} 1 2_4 1 2_2 1_2 2_3 2^* 1 2_2 1_2 2_4 1 2_2 1_2 2_4 1 2_2 1_2 2_2 1 2_4 1_2 \underbrace{2_3 1_3\dots 2_3 1_3}_{a \textrm{ times }}$$
satisfy $\lim\limits_{a\to\infty} \ell(\overline{P_a})=C_{\infty}$. 
  
Indeed, we start by observing that $\lim\limits_{a\to\infty}\lambda_0(\overline{P_a}) = C_{\infty}$ (with the convention that the zeroth position corresponds to $2^*$). Next, we notice that Lemma \ref{l.F2} (15), (16), (19) imply that $\lambda_j(\overline{P_a})<3.118117$ except possibly when the $j=\pm9$ or $0$ modulo the length $P_a$. Moreover, 
\begin{eqnarray*}
\lim\limits_{a\to\infty}\lambda_{-9}(\overline{P_a}) &=& [2;1,2_2,1_2,2_4,1,2_2,1_2,2_4,1,2_2,1_2,2_4,1,2_2,1_2,2_2,1,2_4,1_2,\overline{2_3,1_3}] + [0;2_3,1,\overline{1_3,2_3}] \\ &=& 3.1180041084\dots
\end{eqnarray*}
and 
\begin{eqnarray*}
\lim\limits_{a\to\infty}\lambda_{9}(\overline{P_a}) &=& [2;1,2_2,1_2,2_4,1,2_2,1_2,2_2,1,2_4,1_2,\overline{2_3,1_3}] + [0;2_3,1_2,2_2,1,2_4,1_2,2_2,1,2_4,1,\overline{1_3,2_3}] \\ &=& 3.11812017817071\dots
\end{eqnarray*}
In particular, $\ell(\overline{P_a}) = m(\overline{P_a}) = \lambda_0(\overline{P_a})$ fo all $a\in\mathbb{N}$ sufficiently large. 

In summary, we showed that $\lim\limits_{a\to\infty}\ell(\overline{P_a}) = \lim\limits_{a\to\infty}\lambda_0(\overline{P_a}) = C_{\infty}$, as desired. 
\end{proof}

\end{document}